\documentclass[11pt,a4paper]{amsart}
\usepackage[T1]{fontenc}
\IfFileExists{lmodern.sty}{\usepackage{lmodern}}{}
\usepackage[margin=29mm]{geometry}
\usepackage{amsmath,amssymb,amsthm}
\IfFileExists{microtype.sty}{\usepackage{microtype}}{}
\usepackage{color}
\definecolor{MidnightBlue}{cmyk}{0.98,0.13,0,0.43}
\usepackage[colorlinks=true,linkcolor=MidnightBlue,citecolor=MidnightBlue,urlcolor=MidnightBlue]{hyperref}
\hypersetup{pdftitle={Lawson Homology for Projective Varieties with a C*-Action},pdfauthor={Wenchuan Hu},pdfsubject={Revised manuscript: decompositions, comparison maps, and applications}}
\IfFileExists{enumitem.sty}{%
  \usepackage{enumitem}
  \setlist[enumerate]{label=\textup{(\roman*)},leftmargin=2em}
}{%
  
}
\IfFileExists{needspace.sty}{\usepackage{needspace}}{%
  \newcommand{\Needspace}[1]{\par\pagebreak[3]}
}
\allowdisplaybreaks[1]
\numberwithin{equation}{section}
\newtheorem{theorem}{Theorem}[section]
\newtheorem{proposition}[theorem]{Proposition}
\newtheorem{corollary}[theorem]{Corollary}
\newtheorem{lemma}[theorem]{Lemma}
\theoremstyle{definition}
\newtheorem{definition}[theorem]{Definition}
\newtheorem{example}[theorem]{Example}
\theoremstyle{remark}
\newtheorem{remark}[theorem]{Remark}
\newcommand{\A}{\mathbb A}
\newcommand{\C}{\mathbb C}

\newcommand{\PP}{\mathbb P}
\newcommand{\Q}{\mathbb Q}
\newcommand{\Z}{\mathbb Z}
\newcommand{\LL}{\mathbb L}
\newcommand{\CH}{\operatorname{CH}}
\newcommand{\NS}{\operatorname{NS}}
\newcommand{\Griff}{\operatorname{Griff}}
\newcommand{\Sym}{\operatorname{Sym}}
\newcommand{\rk}{\operatorname{rk}}
\newcommand{\coker}{\operatorname{coker}}
\newcommand{\BM}{\mathrm{BM}}
\newcommand{\pt}{\mathrm{pt}}
\newcommand{\alg}{\mathrm{alg}}
\newcommand{\homol}{\mathrm{hom}}
\newcommand{\cl}{\operatorname{cl}}
\newcommand{\im}{\operatorname{im}}
\newcommand{\gr}{\operatorname{gr}}
\newcommand{\Hdg}{\operatorname{Hdg}}
\newcommand{\Fil}{\mathcal F}
\newcommand{\tors}{\mathrm{tors}}
\newcommand{\PD}{\operatorname{PD}}
\title[Lawson homology and $\C^*$-actions]{Lawson Homology for Projective Varieties with a $\C^*$-Action}
\author{Wenchuan Hu}
\date{\today}

\address{School of Mathematics, Sichuan University, Chengdu 610064, China}
\email{huwenchuan@gmail.com, wenchuan@scu.edu.cn}

\begin{document}
\begin{abstract}
We establish a decomposition of integral Lawson homology compatible with cycle class maps for projective varieties filtered by Zariski locally trivial affine space bundles over smooth projective bases. Resolving the graph closures produces liftings without a flatness assumption on the graph projections. We prove that these liftings are independent of the chosen resolutions and yield compatible projectors and filtrations on Lawson homology and singular homology. When the bases have dimension at most two, the comparison maps are injective, are isomorphisms in degrees $k>2p$, and have torsion free diagonal cokernels determined by the surface bases. In this case all integral Hodge homology classes are algebraic. For smooth projective varieties with a multiplicative group action and fixed components of dimension at most two, the ranks of the diagonal cokernels form a symmetric polynomial. We also describe the change of this polynomial under blow-ups and compute it for rational varieties obtained by blowing up quartic surfaces. We recover the known smooth motivic decomposition and distinguish its consequences from these refinements. Applications include cones, singular hypersurfaces, toric varieties, finite quotients, symmetric products, and Hilbert schemes of points on surfaces with integral Tate motives.
\end{abstract}
\maketitle
{\small\tableofcontents}
\medskip

\section{Introduction and main statements}

In this paper, varieties are reduced complex quasi-projective schemes of finite type; they need not be irreducible. Projectivity and smoothness are stated whenever they are required. For integers $p\geq0$ and $k\geq2p$, Lawson homology is
\[
 L_pH_k(Y)=\pi_{k-2p}\mathcal Z_p(Y),
\]
where $\mathcal Z_p(Y)$ denotes the topological group of algebraic $p$-cycles, with the usual quotient construction for a nonproper variety. Its comparison map (also called the cycle class map) is
\[
 \Phi_{p,k}(Y):L_pH_k(Y)\longrightarrow H_k^{\BM}(Y,\Z).
\]
For projective $Y$, the target is ordinary singular homology. Negative cycle indices are interpreted by stabilization, so that $L_pH_k(Y)=H_k^{\BM}(Y,\Z)$ for $p<0$. Details and conventions are given in Section~\ref{sec:formalism}.

The following formulation allows the ambient variety and the closed terms of the filtration to be singular.

\begin{theorem}[Relative affine decomposition]\label{thm:relative}
Let $X$ be projective, and suppose that
\[
 \varnothing=X_{-1}\subset X_0\subset\cdots\subset X_N=X
\]
is a filtration by closed reduced subvarieties. For each $i$, assume that $W_i=X_i\setminus X_{i-1}$ admits a Zariski locally trivial morphism
\[
 p_i:W_i\longrightarrow F_i
\]
with fiber $\A^{a_i}$, where $F_i$ is smooth, connected, and projective. Then, for all $p\geq0$ and $k\geq2p$, there are isomorphisms
\begin{align}
 \Theta_L:\bigoplus_{i=0}^{N}L_{p-a_i}H_{k-2a_i}(F_i)&\xrightarrow{\ \sim\ }L_pH_k(X),\label{eq:relativeL}\\
 \Theta_H:\bigoplus_{i=0}^{N}H_{k-2a_i}(F_i,\Z)&\xrightarrow{\ \sim\ }H_k(X,\Z),\label{eq:relativeH}
\end{align}
which satisfy
\begin{equation}\label{eq:compat}
 \Phi_{p,k}(X)\,\Theta_L
 =\Theta_H\left(\bigoplus_i\Phi_{p-a_i,k-2a_i}(F_i)\right).
\end{equation}
The graph resolution construction of Section~\ref{sec:proof} gives splittings independent of the chosen resolutions. They are determined by the specified strata and bundle maps, not by $X$ alone. No smoothness or irreducibility assumption on $X_i$ is needed.
\end{theorem}

For disconnected bases, one first separates their connected components into successive strata. The hypotheses above remain unchanged in substance.

\begin{theorem}[Fixed point decomposition; \cite{Brosnan,HuLi,Karpenko}]\label{thm:BB}
Let $X$ be a smooth projective variety with an algebraic $\C^*$-action. Write its fixed locus as $X^{\C^*}=\coprod_iF_i$, and set
\[
 X_i^+=\{x\in X:\lim_{t\to0}t\cdot x\in F_i\},
 \qquad a_i=\dim X_i^+-\dim F_i.
\]
Then \eqref{eq:relativeL}--\eqref{eq:compat} hold with these $F_i$ and $a_i$. Equivalently, $a_i$ is the rank of the positive-weight part of the normal bundle of $F_i$ in $X$.
\end{theorem}

The smooth formula follows from the motivic Bia\l ynicki-Birula decomposition of Karpenko and Brosnan and the action of Chow correspondences on Lawson homology; see \cite[Theorem 6.5]{Karpenko}, \cite[Theorem 3.5]{Brosnan}, and \cite[Theorem 4.3 and Corollary 5.3]{HuLi}. We also give a direct proof through Theorem~\ref{thm:relative}. Compatibility with the comparison map then follows from the functoriality of pushforward and refined pullback.

The smoothness of the bases in Theorem~\ref{thm:relative} has a precise role: after resolving a graph closure, the resulting morphism to $F_i$ is a local complete intersection morphism. Its refined pullback defines the required lifting. Replacing the smoothness of $F_i$ by the irreducibility of the filtration terms does not supply this pullback.

The compatibility statement gives more than an abstract calculation of groups. For example, we will prove the following.

\begin{corollary}\label{cor:intro-small}
Under the hypotheses of Theorem~\ref{thm:BB}, assume that $\dim F_i\leq2$ for every $i$. Then $\Phi_{p,k}(X)$ is injective for all $k\geq2p\geq0$, and is an isomorphism for all $k\geq2p+1$. If every two-dimensional $F_i$ has geometric genus zero, it is an isomorphism also for $k=2p$.
\end{corollary}

This statement concerns cycles modulo algebraic equivalence in degree $2p$. It must not be interpreted as an injectivity statement for Chow groups modulo rational equivalence. In particular, it does not establish Bloch's conjecture for arbitrary surfaces with $p_g=q=0$.

\subsection*{Contribution and relation to earlier work}
The smooth formula above is already a consequence of established motivic theory. The relative cellular methods of Karpenko \cite{Karpenko} and Nenashev--Zainoulline \cite[Theorem 4.4]{NZ} provide the geometric background. The present treatment of a possibly singular total space makes the smooth base hypothesis explicit and supplies a graph resolution proof compatible with the cycle class map. It does not extend that argument to arbitrary singular bases.

The refinements developed here are the resolution independence of the graph liftings (Proposition~\ref{prop:independent}), the resulting compatible projectors and intrinsic filtrations (Theorem~\ref{thm:filtered}), and the integral description of comparison defects for bases of dimension at most two (Theorems~\ref{thm:small} and~\ref{thm:hodge}). Proposition~\ref{prop:defectpoly} expresses the ranks of these cokernels by an action-independent polynomial and proves its symmetry in the smooth fixed point setting. These are structural consequences of the splitting, with proofs supplied below; their presentation here is not a claim that the underlying motivic or surface calculations are new.

The applications retain their established sources. The blow-up formula for Lawson homology was first proved in \cite[Theorems 1.6 and 3.4]{HuBlowup}. Its compatibility with cycle class maps is also stated explicitly by Nie \cite[Proposition 3.8]{Nie}; Hu--Li \cite[Corollary 5.2]{HuLi} later recovered the formula from Chow motives. Proposition~\ref{prop:blowupdefect} and Example~\ref{ex:quarticblowup} derive the blow-up identities for comparison defects and the quartic-surface calculations from these results. Cellular comparison comes from Lima-Filho \cite{LF93}, finite-quotient comparison from Hu--Li \cite{HuLi}, and the Hilbert scheme application from Totaro \cite{Totaro}. The toric computations belong to the broader framework of \cite{FHW}. We give direct proofs where they clarify hypotheses or integral coefficients.

\subsection*{Organization}
Sections~\ref{sec:formalism}--\ref{sec:proof} establish the splitting. Sections~\ref{sec:structure}--\ref{sec:comparison} explain its structure and the low-dimensional consequences. The remaining sections give geometric applications and explicit computations. Throughout, a direct-sum decomposition means an additive decomposition; no multiplicative decomposition is asserted.

\section{Cycle spaces, functoriality, and affine bundles}\label{sec:formalism}

\subsection{Indices and the comparison map}
The constructions of Lawson, Friedlander, Friedlander--Gabber, and Lima-Filho provide proper pushforward, flat pullback, local complete intersection pullback, and localization; see \cite{Friedlander,FG,Lawson,LF92,LF93}. We work integrally unless a tensor product with $\Q$ is displayed, and write $A_\Q=A\otimes_\Z\Q$. No assertion about arbitrary coefficient rings is implicit in this notation. For $p<0$, the stabilized convention is
\begin{equation}\label{eq:negative}
 L_pH_k(Y)=L_0H_{k-2p}(Y\times\A^{-p})
 \cong H_{k-2p}^{\BM}(Y\times\A^{-p},\Z)
 \cong H_k^{\BM}(Y,\Z).
\end{equation}
We put $L_pH_k=0$ when $p\geq0$ and $k<2p$, and put homology in negative degrees equal to zero. Thus, for a point and $k\geq2p$,
\begin{equation}\label{eq:point}
 L_pH_k(\pt)=
 \begin{cases}\Z,&p\leq0,\ k=0,\\0,&\text{otherwise}.\end{cases}
\end{equation}
The negative index extension preserves the functorial operations used below; see \cite[Section 2.1]{HuLi}. It is especially important when $p<a_i$.

For a closed inclusion $Z\hookrightarrow Y$ with complement $U$, localization gives
\begin{equation}\label{eq:localization}
 \cdots\longrightarrow L_pH_{k+1}(U)\xrightarrow{\partial}
 L_pH_k(Z)\longrightarrow L_pH_k(Y)\longrightarrow L_pH_k(U)
 \xrightarrow{\partial}L_pH_{k-1}(Z)\longrightarrow\cdots.
\end{equation}
For $p\geq0$, this sequence terminates with $L_pH_{2p}(U)\to0$. Cycle class maps commute with localization, proper pushforward, flat pullback, and the refined pullbacks used below.

When $Y$ is projective,
\begin{equation}\label{eq:alg}
 L_pH_{2p}(Y)=\CH_p(Y)/\CH_p(Y)_{\alg},
 \qquad
 \ker\Phi_{p,2p}(Y)=\Griff_p(Y),
\end{equation}
where $\CH_p(Y)_{\alg}$ is the subgroup of algebraically trivial Chow classes and $\Griff_p(Y)=\CH_p(Y)_{\homol}/\CH_p(Y)_{\alg}$. In contrast, the negative index convention \eqref{eq:negative} is a convention for Lawson homology, not for ordinary Chow groups: $\CH_q(Y)=0$ for $q<0$.

\subsection{Homotopy invariance}
\begin{lemma}\label{lem:affine}
If $p:W\to F$ is a Zariski locally trivial $\A^a$-bundle, flat pullback induces an isomorphism
\[
 p^*:L_qH_l(F)\xrightarrow{\sim}L_{q+a}H_{l+2a}(W)
\]
for $l\geq2q$, with the preceding conventions. There is a compatible Thom isomorphism on Borel--Moore homology.
\end{lemma}
\begin{proof}
For a product bundle this is the algebraic homotopy theorem. For a general bundle, choose a finite stratification of $F$ subordinate to Zariski trivializations, with successive unions closed. Apply the product result to each stratum and compare the localization sequences for the base and its inverse image. Induction and the five lemma give the assertion; at the terminal degree one uses the corresponding right exact sequences. Formula~\eqref{eq:negative} gives the same assertion for negative indices. The topological assertion follows by the identical localization argument and the complex orientation on the fibers.
\end{proof}

An affine space bundle need not be identified with an algebraic vector bundle. The statement and proof only require the indicated local triviality.

\subsection{Refined pullbacks}
If $g:M\to F$ is a morphism between smooth varieties with $\dim M-\dim F=a$, its graph is a regular embedding in $M\times F$ and the second projection is smooth. Thus $g$ is a local complete intersection morphism of virtual relative dimension $a$, and there is a pullback
\[
 g^!:L_qH_l(F)\longrightarrow L_{q+a}H_{l+2a}(M).
\]
This statement does not assert that $g$ is flat or has equidimensional fibers. On an open set where $g$ is smooth, $g^!$ agrees with flat pullback. It commutes with restriction to open subsets. These are the refined-intersection properties in \cite[Theorem 3.4]{FG}; stabilization and functoriality are given in \cite[Section 2.1 and Proposition 2.2]{HuLi}. Only restriction to open subsets, functoriality of l.c.i. pullback, and the projection formula are used here. No unrestricted base-change identity for an arbitrary Cartesian square is needed.

The corresponding topological map between smooth projective varieties is the Gysin map obtained from Poincar\'e duality and cohomological pullback. The algebraic and topological maps commute with $\Phi$. More explicitly, if $\dim M=m$ and $\dim F=f$, then
\[
 g_H^!=\PD_M^{-1}\,g^*\,\PD_F:
 H_\ell(F,\Z)\longrightarrow H_{\ell+2(m-f)}(M,\Z).
\]
Its restriction to a smooth open subset uses Borel--Moore Poincar\'e duality. The cycle-map compatibility is the usual compatibility of algebraic intersection with topological Gysin maps; see \cite{FG,LF93}.

\section{Proof of the decomposition theorems}\label{sec:proof}

\subsection{A lifting construction}
\begin{lemma}\label{lem:lift}
Let $Y$ be projective, let $W\subset Y$ be open, and let $p:W\to F$ be a Zariski locally trivial $\A^a$-bundle over a smooth connected projective variety. There is a homomorphism
\[
 s:L_{r-a}H_{k-2a}(F)\longrightarrow L_rH_k(Y)
\]
whose restriction to $W$ is $p^*$. The construction also defines a compatible homological lifting $H_{k-2a}(F,\Z)\to H_k(Y,\Z)$.
\end{lemma}
\begin{proof}
Let $\Gamma\subset Y\times F$ be the reduced closure of the graph of $p$. Over $W$, the graph is closed in $W\times F$ because $F$ is separated. Consequently, the inverse image of $W$ under $\Gamma\to Y$ is exactly that graph.

The variety $W$ is smooth. Choose a projective resolution
\[
 \rho:\widetilde\Gamma\longrightarrow\Gamma
\]
that is an isomorphism over $W$. Write
\[
 f:\widetilde\Gamma\longrightarrow Y,
 \qquad g:\widetilde\Gamma\longrightarrow F
\]
for the two projections after resolution. Since $F$ is connected and smooth, it is irreducible; the locally trivial bundle $W$ is also irreducible. Thus $\Gamma$ is the closure of this one graph, even if $Y$ has other components. Both $\widetilde\Gamma$ and $F$ are smooth and projective, and their dimensions differ by $a$. We may therefore set
\begin{equation}\label{eq:lift}
 s=f_*g^!:L_{r-a}H_{k-2a}(F)\longrightarrow L_rH_k(Y).
\end{equation}
Proper pushforward and restriction to $W$ commute. Since $f^{-1}(W)=W$ and $f|_W$ is the identity, while $g|_W=p$, restriction of \eqref{eq:lift} is $p^*$. The same construction in homology has the identical restriction property. Compatibility with cycle classes follows from compatibility of proper pushforward and refined pullback.
\end{proof}

\begin{proposition}[Resolution independence]\label{prop:independent}
For fixed $W\subset Y$ and $p:W\to F$ in Lemma~\ref{lem:lift}, the maps $s_L=f_*g^!$ and $s_H=f_*g_H^!$ are independent of the chosen projective resolution of the graph closure that is an isomorphism over $W$.
\end{proposition}
\begin{proof}
First let $h:M'\to M$ be a proper birational morphism between smooth projective varieties of the same dimension. The projection formula gives, for every Lawson class $\alpha$ in the indicated range,
\[
 h_*h^!\alpha
 =h_*\bigl(h^!\alpha\mathbin\bullet[M']\bigr)
 =\alpha\mathbin\bullet h_*[M']
 =\alpha\mathbin\bullet[M]=\alpha.
\]
Here $[M]$ is the unit for the intersection product, and $h_*[M']=[M]$ since $h$ has degree one. The same identity holds in integral homology. The stabilized projection formula is recorded in \cite[Proposition 2.4(3)]{HuLi}.

Two resolutions $M_1,M_2\to\Gamma$ admit a common smooth projective resolution $M$ of the component of $M_1\times_\Gamma M_2$ containing $W$. Write $h_j:M\to M_j$ for the resulting birational morphisms. If $f=f_jh_j$ and $g=g_jh_j$, functoriality gives
\[
 f_*g^!=f_{j*}h_{j*}h_j^!g_j^!=f_{j*}g_j^!.
\]
Both choices therefore give the same map. The homological proof is identical.
\end{proof}

\begin{proof}[Proof of Theorem~\ref{thm:relative}]
Apply Lemma~\ref{lem:lift} with $Y=X_i$ and $W=W_i$. Denote the resulting lifting by $s_i$. By Lemma~\ref{lem:affine}, the map $p_i^*$ is an isomorphism, so $s_i(p_i^*)^{-1}$ is a right inverse to restriction
\[
 L_pH_k(X_i)\longrightarrow L_pH_k(W_i).
\]
This works in degree $k+1$ as well as in degree $k$. Therefore both boundary maps adjacent to $L_pH_k(X_{i-1})\to L_pH_k(X_i)$ in \eqref{eq:localization} vanish. We obtain a split short exact sequence
\[
 0\longrightarrow L_pH_k(X_{i-1})\longrightarrow L_pH_k(X_i)
 \longrightarrow L_{p-a_i}H_{k-2a_i}(F_i)\longrightarrow0.
\]
At $k=2p$, the right boundary is already zero by the terminal convention; the degree-$k+1$ lifting still proves injectivity on the left. Induction proves \eqref{eq:relativeL}.

The homological liftings give the same split exact sequences in ordinary homology, since the closed terms are projective. This proves \eqref{eq:relativeH}. At each step we use the same resolution and the same proper and Gysin maps in both theories, so the two splittings commute with cycle classes. This proves \eqref{eq:compat}. If $\iota_i:X_i\hookrightarrow X$, the resulting global isomorphism is the sum of $(\iota_i)_*s_i$. Proposition~\ref{prop:independent} proves the asserted independence of resolutions. In fact the closure of the graph in $X\times F_i$ is the same closed subset as its closure in $X_i\times F_i$ followed by inclusion. Hence the global maps depend only on the specified $W_i\subset X$ and $p_i$, and not on a choice of ordering that realizes them as a closed filtration.
\end{proof}

\begin{example}[A graph closure with a larger special fiber]\label{ex:graph}
Let $F=\PP^2$, fix $b\in F$, and put
\[
 Y=\operatorname{Bl}_{(b,\infty)}(F\times\PP^1),
 \qquad W=F\times\A^1\subset Y.
\]
The projection $p:W\to F$ is a trivial affine-line bundle and extends to $g:Y\to F$. Its graph closure is isomorphic to $Y$. A general fiber is $\PP^1$, but $g^{-1}(b)$ contains the exceptional divisor $\PP^2$. Thus the graph projection is not equidimensional of relative dimension one. In particular, taking the closure of the fibers on $W$ does not describe all special fibers of the graph closure. Lemma~\ref{lem:lift} avoids precisely this problem.
\end{example}

\subsection{Actions on smooth and singular varieties}
\begin{proof}[Proof of Theorem~\ref{thm:BB}]
The Bia\l ynicki-Birula theorem gives smooth projective fixed components $F_i$ and Zariski locally trivial affine space bundles $X_i^+\to F_i$. For a projective $X$, these strata can be ordered to form a filtration by closed subvarieties. The attracting fiber dimension is the number of positive normal weights. Apply Theorem~\ref{thm:relative}. The geometric assertions, including filtrability, are proved in \cite{BB73,BB76} and summarized in \cite[Theorem 3.3]{Brosnan}.
\end{proof}

\begin{remark}[Motivic interpretation]\label{rem:motives}
With the homological convention that the Lefschetz motive $\LL$ shifts Lawson indices by $(p,k)\mapsto(p-1,k-2)$, the smooth case has an integral Chow-motive decomposition
\[
 h(X)\cong\bigoplus_i h(F_i)\otimes\LL^{\otimes a_i}.
\]
This is the motivic decomposition of \cite[Theorem 6.5]{Karpenko} and \cite[Theorem 3.5]{Brosnan}; applying \cite[Theorem 4.3]{HuLi} gives Theorem~\ref{thm:BB}. It also gives
\[
 \CH_p(X)\cong\bigoplus_i\CH_{p-a_i}(F_i).
\]
The latter formula uses zero for negative Chow indices. For singular $X$, Theorem~\ref{thm:relative} is a statement about homology theories; we do not assign a smooth-projective Chow motive $h(X)$ to it.
\end{remark}

\begin{corollary}[A sufficient criterion in the singular case]\label{cor:singular}
Let $X$ be projective with a $\C^*$-action. Suppose its attracting sets form a filtration by closed unions and that their limit maps are Zariski locally trivial affine space bundles over smooth projective fixed components. Then the integral decomposition and the compatibility in Theorem~\ref{thm:relative} hold.
\end{corollary}
\begin{proof}
These are exactly the geometric hypotheses of that theorem.
\end{proof}

\begin{remark}
An analytic stratification theorem, such as the results of Carrell--Goresky \cite{CG}, does not by itself verify the stated Zariski local triviality. For a singular action, the algebraic bundle and filtration properties must be checked before applying Corollary~\ref{cor:singular}. Likewise, the argument above does not prove a corresponding theorem with arbitrary singular bases. It would suffice in that situation to construct liftings with the restriction property of Lemma~\ref{lem:lift}; irreducibility alone is not such a construction.
\end{remark}

\section{Projectors, filtrations, and Hodge structures}\label{sec:structure}

Fix the filtration and affine-bundle presentations in Theorem~\ref{thm:relative}. Put
\[
 L_pH_k(Y)_{\homol}=\ker\Phi_{p,k}(Y),\qquad
 I_{p,k}(Y)=\im\Phi_{p,k}(Y),\qquad
 D_{p,k}(Y)=\coker\Phi_{p,k}(Y).
\]
For $p\leq0$, the comparison map $\Phi$ is the Dold--Thom or stabilized identification, so $L_pH_k(Y)_{\homol}$ and $D_{p,k}(Y)$ are zero. All indices below lie in the defining range.

\begin{theorem}[Compatible projectors and filtrations]\label{thm:filtered}
Let $A_i=L_{p-a_i}H_{k-2a_i}(F_i)$ and $B_i=H_{k-2a_i}(F_i,\Z)$. The graph splittings give projectors
\[
 e_i^L=\Theta_L\operatorname{pr}_i\Theta_L^{-1},\qquad
 e_i^H=\Theta_H\operatorname{pr}_i\Theta_H^{-1},
\]
where $\operatorname{pr}_i$ is the endomorphism projecting onto the $i$th summand. They satisfy
\[
 e_i^Le_j^L=\delta_{ij}e_i^L,\qquad \sum_i e_i^L=1,
 \qquad \Phi_{p,k}e_i^L=e_i^H\Phi_{p,k},
\]
and the analogous identities in homology. They are independent of the resolutions.

The filtrations
\begin{align*}
 \Fil_i^L&=\im\bigl(L_pH_k(X_i)\to L_pH_k(X)\bigr),\\
 \Fil_i^H&=\im\bigl(H_k(X_i,\Z)\to H_k(X,\Z)\bigr)
\end{align*}
are intrinsic to the closed filtration. Their associated graded comparison map is canonically identified, via restriction and $p_i^*$, with
\[
 \gr_i\Phi_{p,k}=\Phi_{p-a_i,k-2a_i}(F_i):A_i\longrightarrow B_i.
\]
Moreover $\Phi_{p,k}$ is strict for these filtrations:
\[
 \Phi_{p,k}(\Fil_i^L)=I_{p,k}(X)\cap\Fil_i^H.
\]
Consequently the kernel $L_pH_k(X)_{\homol}$, image $I_{p,k}(X)$, and cokernel $D_{p,k}(X)$ have induced filtrations with associated graded groups
\[
 L_{p-a_i}H_{k-2a_i}(F_i)_{\homol},\qquad
 I_{p-a_i,k-2a_i}(F_i),\qquad D_{p-a_i,k-2a_i}(F_i),
\]
respectively.
\end{theorem}
\begin{proof}
At every stage, the liftings produce a commuting diagram of split short exact sequences. Induction identifies $L_pH_k(X_i)$ with $\bigoplus_{j\leq i}A_j$ and $H_k(X_i)$ with $\bigoplus_{j\leq i}B_j$, compatibly with their inclusions at later stages. In these coordinates $\Phi$ is the direct sum of the comparison maps of the bases. The projector identities and strictness follow immediately. The quotients $\Fil_i/\Fil_{i-1}$ are identified with the stratum groups by restriction, independently of any splitting; composing with the inverse affine-bundle pullback gives the stated canonical identification. Kernels, images and cokernels of the finite direct-sum map give the last assertion. Resolution independence follows from Proposition~\ref{prop:independent}.
\end{proof}

\begin{remark}[The extent of canonicity]
The projectors are determined by the specified bundle presentations. This does not give a preferred decomposition for a variety with no such data, or functoriality for arbitrary morphisms of varieties. Other abstract splittings of the same filtration differ by a filtration-preserving automorphism acting identically on associated graded groups. Its difference from the identity is nilpotent of exponent at most $N+1$. The induced cohomological projectors are generally not multiplicative: the projection onto $H^4(\PP^2,\Z)$ sends $h$ to zero but sends $h^2$ to $h^2$.
\end{remark}

\begin{proposition}[Purity and Tate twists]\label{prop:hodge-split}
The homological splitting of Theorem~\ref{thm:relative}, after tensoring with $\Q$, is an isomorphism of mixed Hodge structures
\begin{equation}\label{eq:hodge-split}
 H_k(X,\Q)\cong\bigoplus_i H_{k-2a_i}(F_i,\Q)(a_i).
\end{equation}
In particular $H_k(X,\Q)$ is pure of weight $-k$, although $X$ may be singular. Its integral torsion is described separately by
\[
 H_k(X,\Z)_{\tors}\cong
 \bigoplus_i H_{k-2a_i}(F_i,\Z)_{\tors}.
\]
\end{proposition}
\begin{proof}
We use Deligne's mixed Hodge structure on homology \cite{Deligne}. In a lifting $f_*g_H^!$, both the source $F_i$ and the resolution $M_i$ are smooth projective. Poincar\'e duality and algebraic pullback show that
\[
 g_H^!:H_{k-2a_i}(F_i,\Q)(a_i)\longrightarrow H_k(M_i,\Q)
\]
is a morphism of Hodge structures. Proper pushforward $f_*$ is a morphism of mixed Hodge structures even when its target is singular. Thus the already established bijection is a bijective morphism in that category. Every summand has weight $-(k-2a_i)-2a_i=-k$, proving purity. The torsion formula follows directly from the integral direct sum; it is not a consequence of rational Hodge theory.
\end{proof}

\section{Comparison maps and low-dimensional fixed components}\label{sec:comparison}

We retain the notation $L_pH_k(-)_{\homol}$ and $D_{p,k}(-)$ from Section~\ref{sec:structure}.

\begin{proposition}\label{prop:defects}
Under the hypotheses of Theorem~\ref{thm:relative},
\begin{align*}
 L_pH_k(X)_{\homol}&\cong\bigoplus_i L_{p-a_i}H_{k-2a_i}(F_i)_{\homol},\\
 D_{p,k}(X)&\cong\bigoplus_i D_{p-a_i,k-2a_i}(F_i).
\end{align*}
For a smooth projective $X$ with a $\C^*$-action this yields
\[
 \Griff_p(X)\cong\bigoplus_{a_i\leq p}\Griff_{p-a_i}(F_i).
\]
\end{proposition}
\begin{proof}
Equation~\eqref{eq:compat} identifies the comparison map with a finite direct sum of comparison maps. Kernels and cokernels commute with this direct sum. The last assertion follows from \eqref{eq:alg}; summands with $p-a_i<0$ contribute zero.
\end{proof}

\begin{lemma}[{Curves and surfaces; \cite[Theorem 4.6]{Friedlander}}]\label{lem:small}
For a smooth projective curve, every Lawson comparison map in the range $k\geq2p$ is an isomorphism. For a smooth projective surface $S$, every comparison map is injective, and the only possible failure of surjectivity is
\[
 \Phi_{1,2}:\NS(S)\hookrightarrow H_2(S,\Z).
\]
Its cokernel is $H_2(S,\Z)/\cl\NS(S)$; it vanishes if $p_g(S)=0$.
\end{lemma}
\begin{proof}
For $p=0$ use Dold--Thom, and for $p<0$ use \eqref{eq:negative}. Top-dimensional cycles give one copy of $\Z$ per connected component in top degree and no higher homotopy. Indices exceeding the dimension contribute zero on both sides in the stated range.

For a surface, the remaining case is $p=1$. The divisor-space calculation \cite[Theorem 4.6]{Friedlander} gives
\[
 L_1H_2(S)=\NS(S),\quad
 L_1H_3(S)\xrightarrow{\sim}H_3(S,\Z),\quad
 L_1H_4(S)\xrightarrow{\sim}H_4(S,\Z),
\]
and zero in higher degrees. The exponential sequence identifies $\NS(S)$ with the kernel of $H^2(S,\Z)\to H^2(S,\mathcal O_S)$, equivalently the integral $(1,1)$-classes, including torsion. Poincar\'e duality gives the asserted injection into $H_2(S,\Z)$. If $p_g(S)=0$, all of $H^2(S,\C)$ has type $(1,1)$, and the integral Lefschetz $(1,1)$ theorem proves surjectivity.
\end{proof}

\begin{theorem}[Comparison over bases of dimension at most two]\label{thm:small}
Suppose all the bases $F_i$ in Theorem~\ref{thm:relative} have dimension at most two. Then
\begin{enumerate}
\item $\Phi_{p,k}(X)$ is injective for every $k\geq2p\geq0$;
\item $\Phi_{p,k}(X)$ is an isomorphism for $k\geq2p+1$;
\item in the remaining degree,
\[
 D_{p,2p}(X)\cong
 \bigoplus_{\substack{\dim F_i=2\\a_i=p-1}}
 H_2(F_i,\Z)/\cl\NS(F_i).
\]
\end{enumerate}
If all surface bases have $p_g=0$, or if all bases have dimension at most one, all comparison maps are isomorphisms in the full range.
\end{theorem}
\begin{proof}
Use Proposition~\ref{prop:defects} and Lemma~\ref{lem:small}. A nonzero cokernel can only have $p-a_i=1$ and $k-2a_i=2$. These equations imply $k=2p$ and $a_i=p-1$. This proves all assertions and also Corollary~\ref{cor:intro-small}.
\end{proof}

\begin{example}
For $X=C\times\PP^m$, with $C$ a smooth projective curve, a diagonal action on $\PP^m$ gives
\[
 L_pH_k(X)\cong\bigoplus_{j=0}^{m}L_{p-j}H_{k-2j}(C)
 \xrightarrow{\sim}H_k(X,\Z).
\]
Nevertheless, if $C$ has positive genus, its degree-zero Chow group is nonzero and maps to zero in homology. In particular, $\CH_0(X)\to H_0(X,\Z)$ is not injective. This illustrates why the Lawson and Chow assertions must be separated.
\end{example}

\subsection{Integral Hodge classes and the diagonal defect}
For a projective $X$ as in Theorem~\ref{thm:relative}, define
\[
 \Hdg_p(X,\Z)=\{\alpha\in H_{2p}(X,\Z):
 \alpha_\C\in H_{2p}(X,\C)^{-p,-p}\}.
\]
This uses Proposition~\ref{prop:hodge-split} and includes all torsion classes. For smooth $X$, Poincar\'e duality identifies this with the usual group of integral Hodge classes in complementary cohomological degree.

\begin{theorem}[Saturated cycle classes]\label{thm:hodge}
Assume that every base in Theorem~\ref{thm:relative} has dimension at most two. Then
\begin{equation}\label{eq:hodge-image}
 \im\bigl(\CH_p(X)\xrightarrow{\cl}H_{2p}(X,\Z)\bigr)
 =I_{p,2p}(X)=\Hdg_p(X,\Z).
\end{equation}
This subgroup is saturated, and
\begin{equation}\label{eq:free-defect}
 D_{p,2p}(X)\cong
 \bigoplus_{\substack{\dim F_i=2\\a_i=p-1}}
 \bigl(H_2(F_i,\Z)/\cl\NS(F_i)\bigr)
 \cong\Z^{\delta_p(X)},
\end{equation}
where
\[
 \delta_p(X)=\sum_{\substack{\dim F_i=2\\a_i=p-1}}
 \bigl(b_2(F_i)-\rho(F_i)\bigr),\qquad \rho(F_i)=\rk\NS(F_i).
\]
In particular the integral Hodge conjecture holds for a smooth projective $X$ with a $\C^*$-action whose fixed components have dimension at most two.
\end{theorem}
\begin{proof}
For a smooth projective surface $S$, the exponential-sequence description in Lemma~\ref{lem:small} embeds
\[
 H^2(S,\Z)/c_1\NS(S)\hookrightarrow H^2(S,\mathcal O_S).
\]
The target is a complex vector space, so the quotient is torsion free. It is finitely generated, hence free of rank $b_2(S)-\rho(S)$. Poincar\'e duality gives the corresponding statement in $H_2(S,\Z)$. This proves \eqref{eq:free-defect} from Theorem~\ref{thm:small}, including saturation.

In degree $2p$, a nonzero summand in \eqref{eq:hodge-split} has $q=p-a_i\geq0$ and is $H_{2q}(F_i,\Q)(a_i)$. For $q=0$ and for top-dimensional cycles, all integral classes are algebraic. The only remaining case is $q=1$ on a surface; there the integral Hodge classes are exactly $\cl\NS(F_i)$. The integral splitting, which is compatible with Hodge structures after rationalization, therefore identifies $I_{p,2p}(X)$ with $\Hdg_p(X,\Z)$, including torsion. Finally $L_pH_{2p}(X)=\CH_p(X)/\CH_p(X)_{\alg}$ and the comparison map is induced by the Chow cycle map. This proves \eqref{eq:hodge-image}.
\end{proof}

For smooth $X$, this theorem also follows from the established integral Chow-motive decomposition and the Lefschetz $(1,1)$ theorem. The conclusion is surjectivity onto integral Hodge classes; the Chow cycle map may still have a nonzero kernel.

For any projective variety $V$, define
\[
 \delta_p(V)=\rk D_{p,2p}(V),\qquad
 \Delta_V(t)=\sum_{p\geq0}\delta_p(V)t^p.
\]
These ranks are finite because $D_{p,2p}(V)$ is a quotient of the finitely generated group $H_{2p}(V,\Z)$; they vanish for $p>\dim V$. Under the hypotheses of Theorem~\ref{thm:hodge}, this definition agrees with the displayed surface formula.

\Needspace{9\baselineskip}
\begin{proposition}[The comparison-defect polynomial]\label{prop:defectpoly}
Under the hypotheses of Theorem~\ref{thm:hodge},
\begin{equation}\label{eq:defectpoly}
 \Delta_X(t)=\sum_{\dim F_i=2}
 \bigl(b_2(F_i)-\rho(F_i)\bigr)t^{a_i+1}.
\end{equation}
The left side depends only on the comparison maps of $X$, so the right side is independent of the chosen admissible presentation. If $X$ is smooth and connected of dimension $n$ and the presentation is the attracting decomposition of a $\C^*$-action, then
\begin{equation}\label{eq:palindrome}
 \Delta_X(t)=t^n\Delta_X(t^{-1}).
\end{equation}
\end{proposition}
\begin{proof}
Equation \eqref{eq:defectpoly} is Theorem~\ref{thm:hodge}, and finiteness follows from the dimension bound on cycles. For the inverse action on a smooth $n$-fold, a surface component $F_i$ has attracting rank $n-2-a_i$: the nonzero normal weights have changed signs. Apply \eqref{eq:defectpoly} to both actions. The contribution $t^{a_i+1}$ becomes $t^{n-1-a_i}$, with the same coefficient, and both sums compute $\Delta_X(t)$. This gives \eqref{eq:palindrome}.
\end{proof}

\begin{example}[A nonzero defect]\label{ex:surfaceproduct}
Let $S\subset\PP^3$ be a smooth quartic surface and let $X=S\times\PP^r$. A diagonal action on the second factor has $r+1$ fixed copies of $S$, with attracting ranks $0,\ldots,r$. Here $b_2(S)=22$: weak Lefschetz gives $b_1=0$, adjunction gives $K_S=0$, and $\int_Sc_2(T_S)=24$ gives the Euler characteristic. Thus
\[
 \Delta_X(t)=(22-\rho(S))(t+t^2+\cdots+t^{r+1}).
\]
Since $h^{2,0}(S)=1$, $\rho(S)\leq20$, and the displayed defects are nonzero. All comparison maps above the cycle degree are nevertheless integral isomorphisms. This separates comparison with all homology from comparison with Hodge classes.
\end{example}

\subsection{Blow-ups}
The Lawson blow-up formula is due to  \cite[Theorems 1.6 and 3.4]{HuBlowup}. The formulation below includes compatibility with cycle class maps; see also Nie \cite[Proposition 3.8]{Nie} and the later motivic proof of Hu--Li \cite[Corollary 5.2]{HuLi}.

\begin{proposition}[{Blow-up formula; \cite[Theorem 3.4]{HuBlowup}}]\label{prop:blowup}
Let $X$ be smooth projective, and let $Z\subset X$ be a smooth closed subvariety of pure codimension $c\geq2$. Put $\widetilde X=\operatorname{Bl}_ZX$. For $k\geq2p\geq0$ there is an isomorphism
\begin{equation}\label{eq:blowup}
 L_pH_k(X)\oplus\bigoplus_{j=1}^{c-1}L_{p-j}H_{k-2j}(Z)
 \xrightarrow{\ \sim\ }L_pH_k(\widetilde X),
\end{equation}
compatible with the corresponding integral homology isomorphism and cycle class maps. Hence
\begin{align*}
 L_pH_k(\widetilde X)_{\homol}
 &\cong L_pH_k(X)_{\homol}\oplus
       \bigoplus_{j=1}^{c-1}L_{p-j}H_{k-2j}(Z)_{\homol},\\
 D_{p,k}(\widetilde X)
 &\cong D_{p,k}(X)\oplus\bigoplus_{j=1}^{c-1}D_{p-j,k-2j}(Z).
\end{align*}
If $\dim Z\leq2$, the kernel is unchanged, and the cokernel is unchanged for $k>2p$. If $\dim Z\leq1$, both are unchanged throughout the stated range.
\end{proposition}
\begin{proof}
Write $\pi:\widetilde X\to X$, $q:E=\PP(N_{Z/X})\to Z$, and $j_E:E\hookrightarrow\widetilde X$, and put $\xi=c_1(\mathcal O_E(1))$. A standard choice of the isomorphism in \eqref{eq:blowup} is
\[
 (\alpha,\beta_1,\ldots,\beta_{c-1})\longmapsto
 \pi^!\alpha+\sum_{j=1}^{c-1}(j_E)_*
       \bigl(\xi^{c-1-j}\cap q^*\beta_j\bigr).
\]
Here $q^*$ raises the cycle dimension by $c-1$, and intersection with $\xi^{c-1-j}$ lowers it by $c-1-j$, so every summand has indices $(p,k)$. The usual stabilized interpretation applies when $p-j<0$. This is the  isomorphism with the exceptional summands indexed by their cycle-dimension shifts. Each operation commutes with the cycle class map, giving the same expression in integral homology; this compatibility is also proved in \cite[Proposition 3.8]{Nie}. Taking kernels and cokernels gives the two decompositions. Lemma~\ref{lem:small} gives the final assertions.
\end{proof}

\begin{proposition}[Comparison defects under blow-ups]\label{prop:blowupdefect}
Under the hypotheses of Proposition~\ref{prop:blowup},
\begin{equation}\label{eq:blowupdefect}
 \Delta_{\widetilde X}(t)
 =\Delta_X(t)+(t+t^2+\cdots+t^{c-1})\Delta_Z(t).
\end{equation}
If $Z$ is a connected surface, put $\tau(Z)=b_2(Z)-\rho(Z)$. Then
\[
 D_{p,2p}(\widetilde X)\cong
 \begin{cases}
 D_{p,2p}(X)\oplus\Z^{\tau(Z)},&2\leq p\leq c,\\
 D_{p,2p}(X),&\text{otherwise},
 \end{cases}
\]
and the additional term in \eqref{eq:blowupdefect} is
$\tau(Z)(t^2+\cdots+t^c)$. In particular, the torsion subgroup of every $D_{p,k}$ is unchanged by a blow-up along a smooth center of dimension at most two. For such centers, all integral Hodge classes on $\widetilde X$ are algebraic if and only if the same holds on $X$.
\end{proposition}
\begin{proof}
Set $k=2p$ in Proposition~\ref{prop:blowup}, take ranks, and sum over $p$. The terms with $p-j<0$ are zero, so the change of variable $q=p-j$ gives \eqref{eq:blowupdefect}. For a surface, Lemma~\ref{lem:small} and its exponential-sequence argument give
\[
 D_{q,2q}(Z)=0\quad(q\ne1),\qquad
 D_{1,2}(Z)\cong\Z^{\tau(Z)}.
\]
Only $j=p-1$ contributes, and $1\leq j\leq c-1$ is equivalent to $2\leq p\leq c$. This also proves the assertion about torsion; disconnected centers are treated componentwise.

The integral homology blow-up isomorphism is compatible, after rationalization, with the Hodge decomposition and its Tate twists. The ordinary Chow blow-up formula \cite[Section 6.7]{FultonIT} is induced by the same algebraic operations. Consequently the quotient of integral Hodge classes by cycle classes is the direct sum of the corresponding quotient for $X$ and the shifted quotients for $Z$. Every integral Hodge class on a smooth projective variety of dimension at most two is algebraic, including torsion, by the Lefschetz $(1,1)$ theorem and the zero- and top-dimensional cases. Thus all additional quotients vanish.
\end{proof}

This consequence of the author's blow-up formula describes the transformation of the comparison-defect polynomial, the torsion in the comparison cokernels, and the obstruction to the integral Hodge conjecture separately.

\begin{example}[Rational varieties with nonzero comparison defects]\label{ex:quarticblowup}
Let $S\subset\PP^3$ be a smooth quartic surface, embed $\PP^3$ linearly in $\PP^n$ with $n\geq4$, and put $X_n=\operatorname{Bl}_S\PP^n$. The variety $X_n$ is smooth, projective, and rational. Since $S$ has codimension $n-2$,
\begin{equation}\label{eq:quarticblowup}
 \Delta_{X_n}(t)=(22-\rho(S))(t^2+t^3+\cdots+t^{n-2}).
\end{equation}
Every $\Phi_{p,k}(X_n)$ is injective, is an isomorphism for $k>2p$, and has a free cokernel for $k=2p$. All integral Hodge classes on $X_n$ are algebraic. These assertions follow from Propositions~\ref{prop:blowup} and~\ref{prop:blowupdefect}, because the comparison maps of $\PP^n$ are isomorphisms. Since $\rho(S)\leq20$, the polynomial is nonzero, although $X_n$ is birational to $\PP^n$. Thus $\Delta_X$ is not a birational invariant.

For $n=4$, the only nonzero comparison cokernel is
\[
 D_{2,4}(X_4)\cong\Z^{22-\rho(S)},\qquad
 L_2H_4(X_4)\cong\Z^{1+\rho(S)},\qquad
 H_4(X_4,\Z)\cong\Z^{23}.
\]
Here $H_2(S,\Z)\cong\Z^{22}$, since a smooth quartic surface is a simply connected $K3$ surface. Equivalently, simple connectedness follows from the Lefschetz hyperplane theorem, and the universal coefficient theorem and Poincar\'e duality then imply the absence of torsion in $H_2(S,\Z)$. The rank calculation was given in Example~\ref{ex:surfaceproduct}. This example has a nonzero cokernel for comparison with all of $H_4$, while comparison with integral Hodge classes is surjective.

It also admits a compatible $\C^*$-action. In coordinates with $\PP^3=\{x_4=0\}$, the action
$t\cdot[x_0:\cdots:x_4]=[x_0:\cdots:x_3:tx_4]$
lifts to $X_4$. Its fixed locus consists of the strict transform $\PP^3$, a copy of $S$, and the point $[0:0:0:0:1]$. Indeed,
$N_{S/\PP^4}\cong\mathcal O_S(4)\oplus\mathcal O_S(1)$
has weights $0$ and $1$; the two fixed sections of its projectivization are respectively contained in the strict transform of $\PP^3$ and equal to the additional copy of $S$. The strict transform of $\PP^3$ is isomorphic to $\PP^3$ because $S$ is a Cartier divisor in it. Its normal weight is $1$, while all tangent weights at the isolated point are $-1$. At the additional surface, the local blow-up chart has normal coordinates $w$ and $z/w$, where $z$ and $w$ originally have weights $0$ and $1$; their weights are $1$ and $-1$. Thus the three attracting ranks are $1$, $1$, and $0$. Theorem~\ref{thm:BB} therefore gives
\[
 L_pH_k(X_4)\cong L_{p-1}H_{k-2}(\PP^3)
 \oplus L_{p-1}H_{k-2}(S)\oplus L_pH_k(\pt),
\]
consistent with the blow-up calculation. This example also shows that the dimension bound on every fixed component in Theorem~\ref{thm:small} is sufficient but not necessary: the three-dimensional fixed component here is a projective space and contributes no comparison defect.
\end{example}

\section{Affine pavings, cones, and explicit hypersurfaces}\label{sec:examples}

\subsection{Affine pavings}
\begin{definition}
An affine paving of a projective variety is a finite filtration by closed subvarieties whose successive differences are affine spaces. We do not require their dimensions to be monotone in the chosen filtration order.
\end{definition}

\begin{corollary}[{Cellular comparison; \cite[Section 5]{LF93}}]\label{cor:cellular}
If $X$ has an affine paving with $c_m$ cells of complex dimension $m$, then for $k\geq2p\geq0$ the cycle map is an isomorphism and
\[
 L_pH_k(X)\cong H_k(X,\Z)\cong
 \begin{cases}\Z^{c_m},&k=2m,\\0,&k\text{ odd}.\end{cases}
\]
Moreover, $\CH_m(X)\to H_{2m}(X,\Z)$ is an isomorphism.
\end{corollary}
\begin{proof}
Apply Theorem~\ref{thm:relative} with all bases points and use \eqref{eq:point}. For the Chow assertion, proceed through the paving. Homological localization is short exact in even degree because odd homology vanishes. For an affine stratum $\A^a$, the group $\CH_p(\A^a)$ is $\Z$ when $p=a$ and zero otherwise; Chow localization is right exact. A diagram chase proves injectivity of the closed pushforward and the middle cycle map, and the closure of a cell supplies the corresponding homological generator. Induction finishes the proof. This recovers the cellular comparison theorem of \cite[Section 5]{LF93}.
\end{proof}

This applies to generalized flag varieties, their Schubert subvarieties, and products of such varieties. It also applies to every Springer fiber for which an affine paving is supplied by a geometric theorem. In particular, it applies to classical Springer fibers by \cite{DLP}, and to type $E_7$ by \cite{DeConciniMaffei}.

\subsection{Projective cones}
For a nonempty projective variety $V\subset\PP^n$, let $\Sigma V\subset\PP^{n+1}$ be its projective cone with vertex $v$.

\begin{proposition}[Suspension formula; \cite{Lawson,LF92}]\label{prop:cone}
For $k\geq2p\geq0$,
\[
 L_pH_k(\Sigma V)\cong
 \begin{cases}
 L_{p-1}H_{k-2}(V),&p>0,\\
 H_{k-2}(V,\Z),&p=0,\ k>0,\\
 \Z,&p=k=0.
 \end{cases}
\]
\end{proposition}
\begin{proof}
For $p>0$, this is the algebraic suspension theorem \cite{Lawson}. It applies even if $V$ is singular. For $p=0$, use Dold--Thom. The complement of the vertex is a line bundle over $V$, and Borel--Moore localization for $\{v\}\subset\Sigma V$ gives the asserted shift for $k\geq2$. The cone is connected, so the map $H_0(\{v\},\Z)\to H_0(\Sigma V,\Z)$ is an isomorphism; the same exact sequence gives $H_1(\Sigma V,\Z)=0$. This proves the remaining cases.
\end{proof}

This calculation follows from the suspension theorem for an arbitrary projective base $V$; it does not require the smooth-base hypothesis of Theorem~\ref{thm:relative}.

\begin{example}[Why smoothness matters for symmetry]
For the smooth quartic surface $S$ of Example~\ref{ex:surfaceproduct}, the cone $X=\Sigma S$ has a two-step filtration: its vertex and an affine-line bundle over $S$. Hence
\[
 \Delta_{\Sigma S}(t)=(22-\rho(S))t^2.
\]
This is not symmetric about $3/2$. Thus \eqref{eq:palindrome} cannot be extended to all singular projective varieties satisfying Theorem~\ref{thm:relative}. The integral splitting and Theorem~\ref{thm:hodge} still apply to this cone.
\end{example}

\subsection{A family with a two-stratum decomposition}
Fix $d\geq1$ and $m\geq1$, and define
\[
 Q_{m,d}=\left\{[x_0:\cdots:x_d:y_0:\cdots:y_d]\in\PP^{2d+1}:
 \sum_{i=0}^{d}x_i^m y_i=0\right\}.
\]

\begin{proposition}\label{prop:hypersurface}
For $k\geq2p\geq0$, the cycle map of $Q_{m,d}$ is an integral isomorphism and
\[
 L_pH_k(Q_{m,d})\cong
 \begin{cases}
 \Z^2,&k=2d,\\
 \Z,&0\leq k\leq4d,\ k\text{ even},\ k\ne2d,\\
 0,&\text{otherwise}.
 \end{cases}
\]
For $m=1$ the variety is a smooth quadric. For $m>1$ its singular locus is $\{x=0\}\cong\PP^d$.
\end{proposition}
\begin{proof}
The closed subvariety $Z=\{x=0\}$ is $\PP^d$. Projection to $[x]$ identifies the open complement with the total space of the rank-$d$ vector bundle
\[
 E_m=\ker\left(\mathcal O_{\PP^d}(1)^{\oplus(d+1)}
 \xrightarrow{(x_0^m,\ldots,x_d^m)}\mathcal O_{\PP^d}(m+1)\right).
\]
The displayed map is surjective since the sections $x_i^m$ have no common zero. Theorem~\ref{thm:relative} gives
\[
 L_pH_k(Q_{m,d})\cong
 L_pH_k(\PP^d)\oplus L_{p-d}H_{k-2d}(\PP^d).
\]
Projective-space homology gives the formula, compatibly with $\Phi$.

The action $t\cdot[x:y]=[x:ty]$ preserves the equation. Its two fixed components are $\{y=0\}$ and $\{x=0\}$, and the open attracting bundle is exactly $E_m\to\{y=0\}$. Finally, the partial derivatives with respect to $y_i$ are $x_i^m$. Their common vanishing forces $x=0$; for $m>1$ the other derivatives also vanish there. For $m=1$ the full gradient never vanishes at a projective point.
\end{proof}

The two summands therefore correspond to the two fixed components, with attracting ranks $0$ and $d$.

\begin{proposition}[The hyperplane class detects information beyond the groups]\label{prop:polarization}
For every $m,d\geq1$, the Chow cycle maps of $Q_{m,d}$ are integral isomorphisms. Let $h=c_1(\mathcal O_{Q_{m,d}}(1))\in H^2(Q_{m,d},\Z)$. Then
\[
 \langle h^{2d},[Q_{m,d}]\rangle=m+1.
\]
Thus the additive Lawson and homology groups are independent of $m$, whereas the cohomology ring together with its hyperplane class and fundamental class detects $m$.
\end{proposition}
\begin{proof}
Use the closed subset $Z=\PP^d$ and its complement $E_m$. Homotopy invariance identifies $\CH_p(E_m)$ with $\CH_{p-d}(\PP^d)$, and their cycle maps to Borel--Moore homology are isomorphisms. The same holds for $Z$. Homological localization is short exact in even degrees, since both pieces have no odd homology. Comparing with right exact Chow localization first proves injectivity of the closed Chow pushforward and then the middle cycle-map isomorphism, by a diagram chase. The top intersection number is the degree of the hypersurface: intersecting with $2d$ general hyperplanes gives a zero-cycle of degree $m+1$. This uses the usual degree and cycle-class compatibility of intersection theory \cite{FultonIT}. The hypersurface is irreducible because the defining polynomial is primitive and linear in the $y_i$, so $[Q_{m,d}]$ is its single fundamental class.
\end{proof}

\section{Bundle decompositions and monodromy}\label{sec:bundles}

\begin{proposition}\label{prop:relativebundle}
Let $B$ be smooth projective and let $E$ be projective. If $E$ has a filtration whose successive differences are Zariski locally trivial $\A^{a_i}$-bundles over $B$, then
\[
 L_pH_k(E)\cong\bigoplus_iL_{p-a_i}H_{k-2a_i}(B),
\]
compatibly with homology. In particular this applies to $B\times Y$ when $Y$ is projective and has an affine paving.
\end{proposition}
\begin{proof}
The first assertion is Theorem~\ref{thm:relative}. For the second, take the product of $B$ with an affine paving of $Y$.
\end{proof}

The standard projective bundle formula has broader hypotheses: for an algebraic vector bundle $V\to B$ of rank $r+1$ on a quasi-projective base,
\begin{equation}\label{eq:projectivebundle}
 L_pH_k(\PP(V))\cong\bigoplus_{j=0}^{r}L_{p-j}H_{k-2j}(B).
\end{equation}
For $p\geq r$, this is the projective bundle theorem of Friedlander--Gabber \cite[Proposition 2.5]{FG}. We  established the formula for all $p\geq0$ over a smooth projective base \cite[Proposition 3.1]{HuBlowup}. The stabilized argument also proves the displayed quasi-projective version: choose $m\geq\max(0,r-p)$, pull $V$ back to $B\times\A^m$, apply the original theorem at cycle index $p+m$, and use homotopy invariance to shift the indices back. This gives
\[
 L_{p+m}H_{k+2m}(\PP(V)\times\A^m)
 \cong\bigoplus_{j=0}^{r}L_{p+m-j}H_{k+2m-2j}(B\times\A^m),
\]
and hence \eqref{eq:projectivebundle}, including summands with $p-j<0$. The splitting is induced by pullback and powers of the global hyperplane class and commutes with the comparison maps; see also \cite[Proposition 2.5]{Nie}. A filtration of $V$ by line subbundles is unnecessary. For a smooth projective curve $C$ of genus $g$ and $d\geq\max(0,2g-1)$, the Abel--Jacobi morphism makes $\Sym^d C$ a projective bundle of relative dimension $d-g$ over $\operatorname{Pic}^d(C)$; choosing a Poincar\'e line bundle and applying Riemann--Roch proves this description. Thus \eqref{eq:projectivebundle} gives its Lawson homology in terms of the Jacobian (see also \cite{MacdonaldCurve}).

\begin{example}[Cellular fibers do not eliminate monodromy]\label{ex:monodromy}
Let $C'$ be an elliptic curve, let $\tau$ be translation by a nonzero point of order two, and let $C=C'/\langle\tau\rangle$. Form
\[
 E=\big(C'\times\PP^1\times\PP^1\big)/\langle\sigma\rangle,
 \qquad \sigma(c,u,v)=(\tau(c),v,u).
\]
The action is free, so $E$ is smooth projective and $E\to C$ is an \'etale locally trivial, hence analytically locally trivial, bundle with cellular fiber $\PP^1\times\PP^1$. However,
\[
 b_2(E)=2,\qquad b_2(C\times\PP^1\times\PP^1)=3.
\]
Indeed, rational cohomology of the free quotient is the invariant part of the cohomology of the cover. In degree two, $H^2(C',\Q)$ contributes one invariant, while the two ruling classes are interchanged and contribute only one. There is no mixed term because $H^1(\PP^1\times\PP^1,\Q)=0$.

Consequently, a product decomposition cannot follow from unspecified local triviality and cellular fibers alone: it already fails for $L_0H_2=H_2$. The sufficient hypothesis of Proposition~\ref{prop:relativebundle} is a global filtration of the indicated kind; analytic local triviality does not imply that hypothesis.
\end{example}

\section{Algebraic tori and toric varieties}

The calculations in this section are included with explicit integral conventions. Friedlander--Haesemeyer--Walker \cite[Sections 5--6]{FHW} place toric varieties in a general theory of refined cycle maps and weight filtrations. Membership in that theory does not say that the unrefined comparison map is an isomorphism for every nonproper toric variety.\label{sec:toric}

\subsection{The punctured-line splitting}
\begin{lemma}\label{lem:punctured}
For a quasi-projective variety $Y$, $p\geq0$, and $k\geq2p$, there is a splitting
\begin{equation}\label{eq:punctured}
 L_pH_k(Y\times\C^*)\cong
 L_{p-1}H_{k-2}(Y)\oplus L_pH_{k-1}(Y).
\end{equation}
The second summand is zero at $k=2p$.
\end{lemma}
\begin{proof}
In $Y\times\PP^1$, the inclusions at $0$ and $\infty$ induce the same pushforward on Lawson homology. Indeed, a continuous path of points in $\PP^1$ and the exterior-product map on cycle spaces give a homotopy between the two pushforwards. Restriction to $Y\times\A^1=Y\times(\PP^1\setminus\{\infty\})$ annihilates the latter inclusion and therefore also the former. Thus the closed pushforward for $Y\times\{0\}\subset Y\times\A^1$ is zero. Localization and homotopy invariance yield
\[
 0\longrightarrow L_{p-1}H_{k-2}(Y)
 \xrightarrow{\pi^*}L_pH_k(Y\times\C^*)
 \xrightarrow{\partial}L_pH_{k-1}(Y)\longrightarrow0.
\]
Let $j_1:Y\hookrightarrow Y\times\C^*$ be the section at $1$. It is a regular embedding, and refined pullback gives
\[
 j_1^!:L_pH_k(Y\times\C^*)\longrightarrow L_{p-1}H_{k-2}(Y),
 \qquad j_1^!\pi^*=\mathrm{id}.
\]
This is a retraction of the left injection. More explicitly, $(j_1^!,\partial)$ identifies the middle group with the sum in \eqref{eq:punctured}: its kernel is zero, and any lift of an element of the second factor can have its first component removed by subtracting $\pi^*j_1^!$. For $p=0$ the pullback is interpreted by stabilization, or equivalently by the Borel--Moore Gysin map. The boundary signs are those of localization. The construction uses refined intersection with a fiber inside the open variety; it does not attempt to intersect cycles on $\C^*$ with the removed point $0$.
\end{proof}

\begin{proposition}\label{prop:torus}
Let $T^n=(\C^*)^n$, with $n\geq1$. For $p\geq0$ and $k\geq2p$,
\begin{equation}\label{eq:torus}
 L_pH_k(T^n)\cong
 \begin{cases}
 \Z^{\binom{n}{k-n}},&n+p\leq k\leq2n,\\
 0,&\text{otherwise}.
 \end{cases}
\end{equation}
For $p=0$ this is the Borel--Moore homology of the torus.
\end{proposition}
\begin{proof}
For $p=0$, Poincar\'e duality and $T^n\simeq(S^1)^n$ give
\[
 H_k^{\BM}(T^n,\Z)\cong H^{2n-k}(T^n,\Z)
 \cong\Z^{\binom n{k-n}}
\]
for $n\leq k\leq2n$, and zero otherwise. For $p>0$, apply Lemma~\ref{lem:punctured} inductively to $T^{n-1}$. The two ranks add by Pascal's identity and are supported exactly in $n+p\leq k\leq2n$. The initial case $n=0$ is \eqref{eq:point}. Since each induction step is a split exact sequence, no torsion extensions occur.
\end{proof}

\begin{example}[Failure of comparison for a nonproper toric variety]\label{ex:toruscounter}
The smooth toric variety $T^2=(\C^*)^2$ has a simplicial fan, but
\[
 L_1H_2(T^2)=0,
 \qquad H_2^{\BM}(T^2,\Q)=\Q.
\]
The first equality follows from \eqref{eq:torus}, or from $\CH_1(T^2)=\operatorname{Pic}(T^2)=0$. The second follows from $H^2(T^2,\Q)=\Q$. Therefore simpliciality alone does not imply rational comparison for a nonproper toric variety. In particular, an arbitrary simplicial toric variety cannot be assumed to have a filtration by quotients of affine spaces with the homological properties used below.
\end{example}

\subsection{Finite quotients and projective simplicial fans}
We use the following finite-quotient theorem of \cite[Proposition 3.1]{HuLi}: if a finite group $G$ acts on a quasi-projective variety $Y$, then proper pushforward induces
\begin{equation}\label{eq:finitequotient}
 (L_pH_k(Y)\otimes\Q)^G\xrightarrow{\sim}L_pH_k(Y/G)\otimes\Q.
\end{equation}
The corresponding statement for Borel--Moore homology is the usual rational finite-quotient theorem. These maps commute with cycle classes. Passing to the effective group removes any kernel of the action.

\begin{lemma}\label{lem:quotientcell}
Suppose $G$ acts linearly on $\A^a$. For $k\geq2p\geq0$, the rational Lawson comparison map of $\A^a/G$ is an isomorphism. Both sides are $\Q$ if $k=2a$ and $p\leq a$, and zero otherwise.
\end{lemma}
\begin{proof}
Homotopy invariance and \eqref{eq:point} calculate the Lawson homology of $\A^a$. Its Borel--Moore homology is $\Z$ in degree $2a$ and zero otherwise. A complex linear transformation preserves the top-dimensional orientation class. Apply \eqref{eq:finitequotient} and its homological counterpart.
\end{proof}

\begin{theorem}\label{thm:simplicial}
If $X$ is a projective simplicial toric variety, then
\[
 \Phi_{p,k}(X)\otimes\Q:L_pH_k(X)\otimes\Q
 \xrightarrow{\sim}H_k(X,\Q)
\]
for every $k\geq2p\geq0$. Odd rational homology vanishes. If $X$ is smooth, the same assertions hold integrally.
\end{theorem}
\begin{proof}
Choose a generic one-parameter subgroup of the dense torus. Its fixed points are the finitely many torus fixed points, and projectivity makes the attracting decomposition filtrable. For a maximal cone $\sigma$, the corresponding affine chart of a simplicial toric variety is $\A^n/G_\sigma$, with $G_\sigma$ a finite diagonal group. After replacing the one-parameter subgroup by a positive power, it lifts to a diagonal action on $\A^n$ with nonzero weights. The attracting set lies in this invariant chart: an orbit whose limit is the fixed point meets the chart and therefore lies in it. After lifting, it is the quotient of the coordinate subspace on which the negative-weight coordinates vanish. Necessity follows by applying the limit to invariant positive powers of each coordinate; sufficiency is immediate for positive weights. Hence each stratum is of the form $\A^{a_i}/G_i$. This is the quotient-cell description in \cite[Section 5.2]{FultonToric}.

By Lemma~\ref{lem:quotientcell}, every stratum has the required rational comparison and has homology only in its even top degree. Homological localization first shows by induction that the closed filtration terms have no odd rational homology. Comparing Lawson and homological localization then gives the comparison isomorphism, including the terminal degree $k=2p$, by induction. In the smooth case the finite quotients are unnecessary: the Bia\l ynicki-Birula strata are affine spaces, so Corollary~\ref{cor:cellular} applies integrally.
\end{proof}

If $X$ has dimension $n$ and $d_i$ is the number of $i$-dimensional cones, the usual fan formula is
\[
 b_{2m}(X)=\sum_{j=m}^{n}(-1)^{j-m}\binom jm d_{n-j}
 \qquad(0\leq m\leq n)
\]
for a projective simplicial toric variety; see \cite[Section 5.2]{FultonToric}. These are the ranks of its rational Lawson groups in degree $2m$ whenever $p\leq m$. The integral rank statement holds for smooth projective toric varieties.

\subsection{Rational and algebraic equivalence on toric varieties}
\begin{theorem}\label{thm:toricchow}
For a quasi-projective toric variety $X$, the natural map
\[
 \rho_X:\CH_p(X)\longrightarrow L_pH_{2p}(X)
\]
is an isomorphism for every $p\geq0$.
\end{theorem}
\begin{proof}
We spell out the localization argument, which does not require a simplicial fan. The higher Chow cycle map \cite[Construction 4.1]{FG}, together with higher Chow localization \cite{Bloch,BlochMoving},
\[
 \CH_p(Y,1)\longrightarrow L_pH_{2p+1}(Y)
\]
is compatible with the boundary maps in the two localization sequences.

First let $Y=T^n$. Homotopy invariance for Chow groups and localization for $T^n\subset\A^n$ give
\[
 \CH_p(T^n)=\begin{cases}\Z,&p=n,\\0,&p\ne n.\end{cases}
\]
Together with \eqref{eq:torus}, this proves that $\rho_{T^n}$ is an isomorphism. The same formula shows that $L_pH_{2p+1}(T^n)$ is zero unless $p=n-1$, in which case it is $\Z^n$. In the latter case,
\[
 \CH_{n-1}(T^n,1)=\mathcal O(T^n)^*
 =\C^*\times\langle z_1,\ldots,z_n\rangle.
\]
Under the higher cycle map, the coordinate units give the $n$ generators. To check the integral normalization, allow the $j$th coordinate to vanish in $\A^1\times T^{n-1}$. The higher Chow boundary of $z_j$ is the fundamental class of the zero fiber with multiplicity one, whereas that of $z_i$ for $i\ne j$ is zero. Compatibility with Lawson localization gives the same boundary matrix. Induction using \eqref{eq:punctured} shows that these $n$ coordinate boundary maps identify $L_{n-1}H_{2n-1}(T^n)$ with $\Z^n$. The images of the units therefore form an integral basis. Thus $\CH_p(T^n,1)\to L_pH_{2p+1}(T^n)$ is surjective for every $p$.

Filter $X$ by the union of torus orbits of dimension at most $i$, denoted $X_{\leq i}$. These unions are closed, and $U_i=X_{\leq i}\setminus X_{\leq i-1}$ is a disjoint union of tori. Compare
\[
 \CH_p(U_i,1)\longrightarrow\CH_p(X_{\leq i-1})
 \longrightarrow\CH_p(X_{\leq i})\longrightarrow\CH_p(U_i)\longrightarrow0
\]
with
\[
 L_pH_{2p+1}(U_i)\longrightarrow L_pH_{2p}(X_{\leq i-1})
 \longrightarrow L_pH_{2p}(X_{\leq i})\longrightarrow L_pH_{2p}(U_i)\longrightarrow0.
\]
The first vertical map is surjective, the second is an isomorphism by induction, and the fourth is an isomorphism by the torus calculation. Hence the images of the first arrows correspond under the second vertical map, and passage to their cokernels followed by the short five lemma proves that the middle vertical map is an isomorphism. This completes the induction.
\end{proof}

This theorem identifies rational and algebraic equivalence in the indicated cycle-theoretic sense. It is not a general assertion that toric Chow groups equal integral singular homology: Example~\ref{ex:toruscounter} already separates these statements.

\subsection{Euler characteristics}
For a toric variety, the groups in fixed cycle dimension are finitely generated and vanish in large degree, by orbit localization and Proposition~\ref{prop:torus}. Define
\[
 \chi_p(X)=\sum_{k\geq2p}(-1)^k\rk L_pH_k(X).
\]

\begin{proposition}\label{prop:euler}
Let $X=X(\Delta)$ be an $n$-dimensional quasi-projective toric variety, and write $d_i=\#\{\sigma\in\Delta:\dim\sigma=i\}$. Then
\[
 \chi_0(X)=d_n,
\]
and, for $1\leq p\leq n$,
\begin{equation}\label{eq:euler}
 \chi_p(X)=\sum_{i=0}^{n-p}d_i(-1)^{n-i+p}
 \binom{n-i-1}{p-1}.
\end{equation}
For $p>n$, $\chi_p(X)=0$.
\end{proposition}
\begin{proof}
Proposition~\ref{prop:torus} gives, for $m\geq p\geq1$,
\[
 \chi_p(T^m)=\sum_{j=p}^{m}(-1)^{m+j}\binom mj
 =(-1)^{m+p}\binom{m-1}{p-1}.
\]
The last identity follows by Pascal's identity and telescoping. Also $\chi_0(T^m)=0$ for $m>0$, while $\chi_0(T^0)=1$. Localization, including its terminal degree, makes $\chi_p$ additive for a closed subvariety and its open complement. The orbit of a cone of dimension $i$ has dimension $n-i$, so summing these contributions proves \eqref{eq:euler} and the other assertions.
\end{proof}

\section{Symmetric products and Hilbert schemes}\label{sec:symmetric}

\subsection{Finite quotients of cellular varieties}
\begin{proposition}\label{prop:quotient}
Let $Y$ be a projective variety with an affine paving, and let a finite group $G$ act algebraically on $Y$. Then every rational Lawson comparison map of $Y/G$ is an isomorphism. In particular, for $d\geq0$,
\[
 L_pH_k(\Sym^dY)\otimes\Q\xrightarrow{\sim}H_k(\Sym^dY,\Q)
 \qquad(k\geq2p\geq0),
\]
and the odd rational homology of $\Sym^dY$ vanishes.
\end{proposition}
\begin{proof}
Corollary~\ref{cor:cellular} identifies the comparison map of $Y$ with an isomorphism. This map is equivariant under $G$. Apply \eqref{eq:finitequotient} and its homological counterpart. For the symmetric product, $Y^d$ has the product affine paving and $\Sym^dY=Y^d/\mathfrak S_d$. The rational cohomology or homology of the quotient has no odd part because that of $Y^d$ has none.
\end{proof}

The rational coefficient restriction is essential to this proof: taking invariants and dividing by the finite-group order do not provide an integral comparison theorem. No integral injectivity or surjectivity claim about symmetric products is made here.

Put $P(V;t)=\sum_k b_k(V)t^k$, where $b_k(V)=\dim_\Q H_k(V,\Q)$, and write $c_j=b_{2j}(Y)$. Macdonald's formula \cite{Macdonald} becomes
\begin{equation}\label{eq:macdonald}
 \sum_{d\geq0}P(\Sym^dY;t)u^d
 =\prod_{j\geq0}(1-u t^{2j})^{-c_j}.
\end{equation}
It also follows directly by taking symmetric powers of the even graded rational homology of $Y$. Thus
\[
 \dim_\Q L_pH_{2m}(\Sym^dY)_\Q
 =[u^dv^m]\prod_{j\geq0}(1-u v^j)^{-c_j}
 \quad(m\geq p).
\]

\begin{example}\label{ex:SP}
For $Y=\PP^n$, these dimensions count partitions of $m$ into at most $d$ parts, each at most $n$. For example,
\[
 P(\Sym^2\PP^2;t)=1+t^2+2t^4+t^6+t^8.
\]
For $\Sym^6\PP^3$, the even Betti numbers, indexed by $m=0,\ldots,18$, are
\[
 (1,1,2,3,4,5,7,7,8,8,8,7,7,5,4,3,2,1,1).
\]
Their sum is $84=\binom96$. In particular,
\[
 L_{16}H_{32}(\Sym^6\PP^3)_\Q\cong H_{32}(\Sym^6\PP^3,\Q)\cong\Q^2.
\]
This rational calculation does not resolve the integral comparison map or its torsion.
\end{example}

\subsection{Hilbert schemes and integral Tate motives}
\begin{lemma}[Standard consequence of an integral Tate decomposition]\label{lem:tate}
Let $M$ be a smooth projective complex variety whose integral Chow motive is
\[
 h(M)\cong\bigoplus_{j\geq0}(\LL^{\otimes j})^{\oplus b_j}.
\]
Then every integral Lawson comparison map of $M$ is an isomorphism in the range $k\geq2p\geq0$. Moreover $H_{2j}(M,\Z)=\Z^{b_j}$, odd homology vanishes, and $\CH_j(M)\to H_{2j}(M,\Z)$ is an isomorphism.
\end{lemma}
\begin{proof}
Apply the Lawson, Chow, and integral homology realizations to the displayed motivic isomorphism. The comparison maps commute with the correspondences defining it. The assertion follows from the calculation for a point and the Tate shift convention in Remark~\ref{rem:motives}.
\end{proof}

\begin{theorem}[Consequence of Totaro \cite{Totaro}]\label{thm:hilbert}
Let $S$ be a smooth projective complex surface whose integral Chow motive is a finite direct sum of Tate motives. For every $d\geq0$,
\[
 L_pH_k(S^{[d]})\xrightarrow{\sim}H_k(S^{[d]},\Z)
 \qquad(k\geq2p\geq0).
\]
In particular, this holds for every smooth projective rational surface $S$. All these groups are free abelian, and odd-degree groups vanish.
\end{theorem}
\begin{proof}
Totaro's theorem \cite[Theorem 4.1]{Totaro} states that an integral Tate decomposition for $S$ implies an integral Tate decomposition for $S^{[d]}$. Apply Lemma~\ref{lem:tate}. Smooth projective rational surfaces have integral Tate motives, for example by the classification of rational surfaces and the integral blow-up formula at points. This gives the final assertion for all rational surfaces, without assuming a torus action on $S$.
\end{proof}

For a connected surface as in Theorem~\ref{thm:hilbert}, put $b=b_2(S)$. The ranks are obtained from G\"ottsche's formula \cite{Goettsche}:
\begin{equation}\label{eq:goettsche}
 \sum_{d\geq0}P(S^{[d]};t)u^d
 =\prod_{\ell\geq1}
 \frac{1}{(1-t^{2\ell-2}u^\ell)(1-t^{2\ell}u^\ell)^b
 (1-t^{2\ell+2}u^\ell)}.
\end{equation}
The coefficient of $t^ku^d$ is the rank of $L_pH_k(S^{[d]})$ whenever $k\geq2p$.

\begin{example}
For $S=\PP^2$ and $d=2$, \eqref{eq:goettsche} gives
\[
 P((\PP^2)^{[2]};t)=1+2t^2+3t^4+2t^6+t^8.
\]
Thus the integral Lawson group in degree four is $\Z^3$ for $p=0,1,2$, whereas the corresponding rational group of $\Sym^2\PP^2$ has rank two. The exceptional geometry of the Hilbert--Chow morphism accounts for this distinction.
\end{example}

\begin{remark}
The integral conclusion for all rational surfaces cannot be justified by asserting that every such surface carries an effective two-dimensional torus action. For instance, the blow-up of $\PP^2$ at four points in general position is a degree-five del Pezzo surface with finite automorphism group. One can see finiteness by its action on the ten $(-1)$-curves: an automorphism fixing them all descends to an automorphism of $\PP^2$ fixing a projective frame, and is the identity. Totaro's theorem gives the integral conclusion without such an action. It also shows that the relevant hypothesis is the integral Tate motive, which is stronger than merely having algebraic cohomology.
\end{remark}

\end{document}